\documentclass[11pt,a4paper]{article}
\usepackage{hyperref}
\hypersetup{	
	colorlinks,
	linkcolor=black,
	citecolor=black,
	urlcolor=black,
	breaklinks=true
}
\usepackage{ucs}
\usepackage[applemac]{inputenc}
\usepackage[T1]{fontenc}
\usepackage[english]{babel}
\usepackage{
	amsmath,
	amsthm,
	amssymb,
	amsfonts,
	mathtools,
	todonotes
}
\DeclareMathAlphabet{\mathpzc}{OT1}{pzc}{m}{it}
\usepackage{geometry}
\newtheoremstyle{lemma}{\topsep}{\topsep}
	{\itshape}
	{}
	{\bfseries}
	{.}
	{\newline}
	{\thmname{#1}\thmnumber{ #2}\thmnote{ (#3)}}	
\theoremstyle{lemma}
	\newtheorem{lemma}{Lemma}  
	\newtheorem{theorem}[lemma]{Theorem}
	\newtheorem{proposition}[lemma]{Proposition}
	\newtheorem{corollary}[lemma]{Corollary}  
	\newtheorem*{speciallemma*}{Lemma 5}
	\newtheorem*{specialproposition*}{Proposition 3}
	\newtheorem*{specialcorollary*}{Corollary 4}

\newtheoremstyle{definition}{\topsep}{\topsep}
	{}
	{}
	{\bfseries}
	{.}
	{\newline}
	{\thmname{#1}\thmnumber{ #2}\thmnote{ (#3)}}	
\theoremstyle{definition}

\newcommand{\N}{\ensuremath{\mathbb{N}}}

\newcommand{\R}{\ensuremath{\mathbb{R}}}

\newcommand{\sphere}{\ensuremath{\mathbb{S}}}

\newcommand{\HM}{\ensuremath{\mathcal{H}}}

\DeclarePairedDelimiter\abs{\lvert}{\rvert}
\DeclarePairedDelimiter\norm{\lVert}{\rVert}

\newcommand{\dd}{\ensuremath{\mathrm{d}}}

\DeclareMathOperator{\dist}{dist}

\newcommand{\gConst}{\ensuremath{C_{\mathrm{g}}}}
\newcommand{\bConst}{\ensuremath{C_{\mathrm{b}}}}

\renewcommand{\phi}{\varphi}
\renewcommand{\epsilon}{\varepsilon}

\hypersetup{
	pdftitle={Discrete Moebius Energy},
	pdfsubject={},
	pdfauthor={Sebastian Scholtes},
	pdfkeywords={},
	pdfcreator={},
	pdfproducer={}
}

\makeindex
\begin{document}

\title{Discrete M\"obius Energy}
\author{\href{mailto:sebastian.scholtes@rwth-aachen.de}{Sebastian Scholtes}}
\date{\today}
\maketitle

\begin{abstract}
	We investigate a discrete version of the M\"obius energy, that is of geometric interest in its own right and is defined on equilateral polygons with $n$ segments. 
	We show that the $\Gamma$-limit regarding $L^{q}$ or $W^{1,q}$ convergence, $q\in [1,\infty]$ of these energies as $n\to\infty$ is the smooth M\"obius energy.  
	This result directly implies the convergence of almost minimizers of the discrete energies to minimizers of the smooth energy 
	if we can guarantee that the limit of the discrete curves belongs to the same knot class. Additionally, we show that the unique minimizer amongst all polygons is the regular $n$-gon. 
	Moreover, discrete overall minimizers converge to the round circle.
\end{abstract}
\centerline{\small Mathematics Subject Classification (2010): 49J45; 57M25, 49Q10, 53A04}

\section{Introduction}

The \emph{M\"obius energy} of a closed rectifiable curve $\gamma$ of length $L$ that is parametrised by arc length is given by 
\begin{align*}
	\mathcal{E}(\gamma)\vcentcolon=\int_{\sphere_{L}}\int_{\sphere_{L}}\Big(\frac{1}{\abs{\gamma(t)-\gamma(s)}^{2}}-\frac{1}{d(t,s)^{2}}\Big)\,\dd s\,\dd t.
\end{align*}
Here, $\sphere_{L}$ is the circle of length $L$ and $d$ the intrinsic metric on $\sphere_{L}$. This energy was introduced by O'Hara in \cite{OHara1991a} and has the interesting property
that it is invariant under M\"obius transformations, hence its name. O'Hara could show that finite energy prevents the curve from having selfintersections. Later on, the existence of energy minimizers in
prime knot classes was proven by Freedman, He and Wang in \cite{Freedman1994a}, while due to computer experiments done in \cite{Kusner1997a} there is a folklore conjecture, usually attributed to Kusner and Sullivan, 
that questions the existence in composite knot classes.
Additionally, it was shown in \cite{Freedman1994a} that the unique absolute minimizer is the round circle, see \cite{Abrams2003a} for this result for a broader class of energies.
In \cite{Freedman1994a,He2000a,Reiter2010a,Reiter2012a,Blatt2012b} the regularity of minimizers and, more generally, of critical points is investigated and it could be proved that they are smooth.
Furthermore, it was shown in \cite{Blatt2012a} that the M\"obius energy of a curve is finite if and only if the curve is embedded and the arc length parametrisation belongs to the fractional Sobolev space 
$W^{1+\frac{1}{2},2}(\sphere_{L},\R^{d})$.\\

When searching for discrete analogs to smooth problems it is not only of interest to approximate the smooth problem in a such a way that numerical computations may be done
more efficiently, but, more importantly, the discrete problem should reflect the structure of the smooth problem and be of interest in its own right.
A discrete version of the M\"obius energy, called \emph{minimum distance energy}, was introduced in \cite{Simon1994a}. If $p$ is a polygon with $n$ consecutive segments
$X_{i}$ this energy\footnote{Actually, this is a minor variant which is more commonly used than the energy originally considered in \cite{Simon1994a}.} is defined by
\begin{align}\label{MDenergy}
	\begin{split}
		\mathcal{E}_{\text{md},n}(p)\vcentcolon=\mathcal{U}_{\text{md},n}(p)-\mathcal{U}_{\text{md},n}(g_{n})\quad\text{with}\quad
		\mathcal{U}_{\text{md},n}(p)\vcentcolon=\sum_{i=1}^{n}\sum_{\substack{j=1\\X_{i},X_{j}\text{ not adjacent}}}^{n}\frac{\abs{X_{i}}\abs{X_{j}}}{\dist(X_{i},X_{j})^{2}},
	\end{split}
\end{align}
where $g_{n}$ is the regular $n$-gon. Note, that this energy is scale invariant. In \cite{Rawdon2006a,Rawdon2010a} questions regarding the relation between the minimum distance energy
and the M\"obius energy were considered. It could be shown that the minimum distance energy $\mathcal{E}_{\text{md},n}$ of polygons that are suitably inscribed in a $C^{2}$ knot $\gamma$ converges 
to the M\"obius energy $\mathcal{E}(\gamma)$ as $n\to\infty$.
Furthermore, an explicit error bound on the difference between the minimum distance energy of an equilateral polygonal knot
and the M\"obius energy of a smooth knot, appropriately inscribed in the polygonal knot, could be established in terms of thickness and the number of segments. 
However, it is not possible to infer from these results that the minimal minimum distance
energy converges to the minimal M\"obius energy in a fixed knot class. 
In \cite{Speller2007a} it was shown that the overall minimizers of the minimum distance energy must be convex and from \cite{Tam2006a,Speller2008a} we know that the overall minimizers in the class of
$4$- and $5$-gons are the regular $4$- and $5$-gon, respectively. This evidence supports the conjecture that the regular $n$-gon minimizes the minimum distance energy in the class of $n$-gons.
Numerical experiments regarding the minimum distance energy under the elastic flow were realized in \cite{Hermes2008a}.\\

In the present note we use another, more obvious, discrete version of the M\"obius energy, that was also used for numerical experiments in \cite{Kim1993a}.
This energy, defined on the class of arc length parametrisations of polygons of length $L$ with $n$ segments, is given by
\begin{align}\label{discreteMoebiusenergy}
	\mathcal{E}_{n}(p)\vcentcolon=\sum_{\substack{i,j=1\\i\not=j}}^{n}\Big(\frac{1}{\abs{p(a_{j})-p(a_{i})}^{2}}-\frac{1}{d(a_{j},a_{i})^{2}}\Big)d(a_{i+1},a_{i})d(a_{j+1},a_{j}),
\end{align}
where the $a_{i}$ are consecutive points on $\sphere_{L}$, or the interval $[0,L]$ if we think of the polygon as being parametrised over an interval.
This energy is scale invariant. A slight variant would be to take $2^{-1}(d(a_{k-1},a_{k})+d(a_{k},a_{k+1}))$ instead of $d(a_{k+1},a_{k})$.
As for the minimum distance energy we are interested
whether polygonal minimizers of \eqref{discreteMoebiusenergy} in a fixed tame knot class $\mathcal{K}$ converge to a minimizer of the M\"obius energy in a suitable topology. The following theorem
reveals the relationship of discrete and smooth M\"obius energy in terms of the so-called $\Gamma$-convergence invented by DeGiorgi. In order to establish \emph{$\Gamma$-convergence}, we have to verify 
two inequalities called $\liminf$ inequality, see Proposition \ref{liminfmoebius}, and $\limsup$ inequality, see Proposition \ref{limsupmoebius}.

\begin{theorem}[M\"obius energy $\mathcal{E}$ is $\Gamma$-limit of discrete M\"obius energies $\mathcal{E}_{n}$]\label{maintheoremgammamoebius}
	For $q\in [1,\infty]$, $\norm{\cdot}\in\{\norm{\cdot}_{L^{q}(\sphere_{1},\R^{d})},\norm{\cdot}_{W^{1,q}(\sphere_{1},\R^{d})}\}$ and every tame knot class $\mathcal{K}$ holds 
	\begin{align*}
		\mathcal{E}_{n}\xrightarrow{\Gamma} \mathcal{E}\quad\text{on }(\mathcal{C}_{1,\text{p}}(\mathcal{K}),\norm{\cdot}).
	\end{align*}
\end{theorem}

Here, $\mathcal{C}_{1,\text{p}}\vcentcolon=\mathcal{C}\cap (C^{1}\cup \bigcup_{n\in\N}\mathcal{P}_{n})$, where $\mathcal{C}$ is the space of arc length curves of length $1$ and $\mathcal{P}_{n}$ the subspace
of equilateral polygons with $n$ segments. Adding a knot class $\mathcal{K}$ in brackets to a set of curves restricts this set to the subset of curves that belong to the knot class $\mathcal{K}$.
By $W^{k,q}(\sphere_{1},\R^{d})$ we denote the standard Sobolev spaces of $k$-times weakly differentiable closed curves with $q$-integrable weak derivative.
The functionals are extended by infinity outside their natural domain.
The notion of $\Gamma$-convergence is devised in such a way, as to allow the convergence of minimizers and even almost minimizers,
see \cite[Corollary 7.17, p.78]{Dal-Maso1993a}. Considering the fact that the $\liminf$ inequality holds on $\mathcal{C}(\mathcal{K})$ and that we already know that minimizers of $\mathcal{E}$ in prime knot classes
are smooth Lemma \ref{variationalconvergencegeneral} yields:

\begin{corollary}[Convergence of discrete almost minimizers]\label{corollaryconvergenceminmizers}
	Let $\mathcal{K}$ be a tame prime knot class, $p_{n}\in\mathcal{P}_{n}(\mathcal{K})$ with
	\begin{align*}
		\abs{\inf_{\mathcal{P}_{n}(\mathcal{K})}\mathcal{E}_{n}-\mathcal{E}_{n}(p_{n})}\to 0\qquad\text{and}\qquad
		p_{n}\to\gamma\in\mathcal{C}(\mathcal{K})\text{ in }L^{1}(\sphere_{1},\R^{d}).
	\end{align*}
	Then $\gamma$ is a minimizer of $\mathcal{E}$ in $\mathcal{C}(\mathcal{K})$ and $\lim_{k\to\infty}\mathcal{E}_{n}(p_{n})=\mathcal{E}(\gamma)$.
\end{corollary}

This result remains true for subsequences, where the number of edges is allowed to increase by more than $1$ for two consecutive polygons.
Since all curves are parametrised by arc length it is not hard to find a
subsequence of the almost minimizers that converges in $C^{0}$, but generally this does not guarantee that the limit curve belongs to the same knot class or is parametrised by arc length.

\begin{proposition}[Order of convergence for M\"obius energy of inscribed polygon]\label{orderofconvergence}
	Let $\gamma\in C^{1,1}(\sphere_{L},\R^{d})$ be parametrised by arc length, $c,\overline c>0$. Then for every $\epsilon\in (0,1)$ there is a $C_{\epsilon}>0$ such that 
	\begin{align*}
		\abs{\mathcal{E}(\gamma)-\mathcal{E}_{n}(p_{n})}\leq \frac{C_{\epsilon}}{n^{1-\epsilon}}
	\end{align*}
	for every inscribed polygon $p_{n}$ given by a subdivision $b_{k}$, $k=1,\ldots,n$ of $\sphere_{L}$ such that 
	\begin{align*}
		\frac{c}{n}\leq \min_{k=1,\ldots,n}\abs{\gamma(b_{k+1})-\gamma(b_{k})} \leq \max_{k=1,\ldots,n}\abs{\gamma(b_{k+1})-\gamma(b_{k})}\leq \frac{\overline c}{n}.
	\end{align*}
\end{proposition}

This is in accordance with the data from \cite{Kim1993a}, which suggests that the order of convergence should be roughly $1$. If we do not assume any regularity we might not be able to control the order of
convergence but we still know that the energies converge:

\begin{corollary}[Convergence of M\"obius energies of inscribed polygons]\label{convergencemoebiusinscribedpolygons}
	Let $\gamma\in\mathcal{C}$ with $\mathcal{E}(\gamma)<\infty$ and $p_{n}$ as in Proposition \ref{orderofconvergence}. Then $\lim_{n\to\infty}\mathcal{E}_{n}(p_{n})=\mathcal{E}(\gamma)$.
\end{corollary}

Using the results from \cite{Gabor1966a}\footnote{It seems that given name and family name of the author of this paper are interchanged, as L{{\"u}}k{\H{o}} is the family name, see also \cite[Footnote 4]{Abrams2003a}}
that are also immanent in \cite{Abrams2003a} we easily get the following result concerning discrete minimizers for all $n\in\N$ in contrast to the situation for the minimum distance energy, where
the analogous result is by now only known for $n\leq 5$, see \cite{Tam2006a,Speller2008a}.

\begin{lemma}[Regular $n$-gon is unique minimizer of $\mathcal{E}_{n}$ in $\mathcal{P}_{n}$]\label{discreteminimizermoebius}
	The unique minimizer of $\mathcal{E}_{n}$ in $\mathcal{P}_{n}$ is a regular $n$-gon.
\end{lemma}

This directly yields the convergence of overall discrete minimizers to the circle, which according to \cite[Corollary 2.2]{Freedman1994a}, is the overall minimizer in $\mathcal{C}$.

\begin{corollary}[Convergence of discrete minimizers to the round circle]\label{convergenceofdiscreteminimizerstoroundcircle}
	Let $p_{n}\in\mathcal{P}_{n}$ bounded in $L^{\infty}$ with $\mathcal{E}_{n}(p_{n})=\inf_{\mathcal{P}_{n}}\mathcal{E}_{n}$.
	Then there is a subsequence with $p_{n_{k}}\to \gamma$ in $W^{1,\infty}(\sphere_{1},\R^{d})$, where $\gamma$ is a round unit circle.
\end{corollary}

One of the main differences between the discrete M\"obius energy \eqref{discreteMoebiusenergy} and the minimum distance energy \eqref{MDenergy} is that bounded minimum distance energy avoids double point singularities, 
while for \eqref{discreteMoebiusenergy} this is only true in the limit. This avoidance of singularities enables the proof of the existence of minimizers of the minimum distance energy
\eqref{MDenergy} via the direct method, see \cite{Simon1994a}. This might be harder or even impossible to achieve for the energy \eqref{discreteMoebiusenergy}. Nevertheless, the relation between the discrete 
M\"obius energy \eqref{discreteMoebiusenergy} and the smooth M\"obius energy is more clearly visible than for the minimum distance energy \eqref{MDenergy}, as reflected in Theorem \ref{maintheoremgammamoebius} and 
Corollaries \ref{corollaryconvergenceminmizers}-\ref{convergenceofdiscreteminimizerstoroundcircle}.
\\

\textbf{Acknowledgement}\\
The author thanks H. von der Mosel, for his interest and many useful suggestions and remarks.
Additionally, we are grateful to the anonymous referee, who was very patient with us and improved the proof for Proposition \ref{liminfmoebius} considerably.
Moreover, we thank Ph. Reiter for an enquiry that led to Proposition \ref{orderofconvergence}.
Furthermore, we thank the participants of CAKE $14$ for helpful conversations, especially S. Blatt, E. Denne, Ph. Reiter, A. Schikorra and M. Szuma{\'n}ska.

\section{The $\liminf$ inequality}

To keep notation simple we accommodate for the fact that we deal with closed curves by indicating the shortest distance metric on the circle $\sphere_{1}$ and the corresponding metric on the unit interval 
both by $\abs{\cdot-\cdot}$, in a similar manner we proceed when comparing indices and computing integrals and sums. By $\gConst$ we denote a generic constant that may change from line to line.

\begin{proposition}[The $\liminf$ inequality]\label{liminfmoebius}
	Let $p_{n},\gamma\in\mathcal{C}$ with $p_{n}\to \gamma$ in $L^{1}(\sphere_{1},\R^{d})$. Then
	\begin{align*}
		\mathcal{E}(\gamma)\leq \liminf_{n\to\infty}\mathcal{E}_{n}(p_{n}).
	\end{align*}
\end{proposition}
\begin{proof}
	We assume that the $\liminf$ is finite and thus $p_{n}\in\mathcal{P}_{n}$ for a subsequence.
	Now the proposition is a consequence of Fatou's Lemma, since for a further subsequence and $s\not=t$ holds
	\begin{align*}
		\sum_{\substack{i,j=1\\i\not=j}}^{n}\Big(\frac{1}{\abs{p(a_{j})-p(a_{i})}^{2}}-\frac{1}{\abs{a_{j}-a_{i}}^{2}}\Big)\chi_{[a_{i},a_{i+1})\times[a_{j},a_{j+1})}(s,t)
		\to \frac{1}{\abs{\gamma(t)-\gamma(s)}^{2}}-\frac{1}{\abs{t-s}^{2}}.
	\end{align*}
\end{proof}

This means, for polygons $p_{n}\in\mathcal{P}_{n}$ that are bounded in $L^{\infty}$ and have equibounded energies, we find a subsequence that converges in $C^{0}$ 
to an embedded $W^{1+\frac{1}{2},2}$ curve with finite M\"obius energy.

\section{Approximation of energy for inscribed polygons}

\begin{specialproposition*}[Order of convergence for M\"obius energy of inscribed polygon]
	Let $\gamma\in C^{1,1}(\sphere_{L},\R^{d})$ be parametrised by arc length, $c,\overline c>0$. Then for every $\epsilon\in (0,1)$ there is a $C_{\epsilon}>0$ such that 
	\begin{align*}
		\abs{\mathcal{E}(\gamma)-\mathcal{E}_{n}(p_{n})}\leq \frac{C_{\epsilon}}{n^{1-\epsilon}}
	\end{align*}
	for every inscribed polygon $p_{n}$ given by a subdivision $b_{k}$, $k=1,\ldots,n$ of $\sphere_{L}$ such that 
	\begin{align}\label{segmentsareapproximatelyequal}
		\frac{c}{n}\leq \min_{k=1,\ldots,n}\abs{\gamma(b_{k+1})-\gamma(b_{k})} \leq \max_{k=1,\ldots,n}\abs{\gamma(b_{k+1})-\gamma(b_{k})}\leq \frac{\overline c}{n}.
	\end{align}
\end{specialproposition*}
\begin{proof}
	By \cite[Lemma 2.1]{Blatt2012a} we know that there is a constant $\bConst>0$ with $\abs{t-s}\leq \bConst\abs{\gamma(t)-\gamma(s)}$. Choose $N\vcentcolon= 4\bConst\,\frac{\overline c}{c}$.
	Let $p$ be an inscribed polygon given by $\gamma(b_{i})$ with $\gamma(b_{0})=\gamma(b_{n})$, $\gamma(b_{1})=\gamma(b_{n+1})$, $b_{0}=0$. Then the arc length parameters of the polygon are 
	$a_{i}=\sum_{k=1}^{i}\abs{\gamma(b_{k})-\gamma(b_{k-1})}$.
	Note, that 
	\begin{align}\label{inequaltyab}
		\abs{b_{j}-b_{i}}\geq \abs{a_{j}-a_{i}}=\sum_{k=\min\{i+1,j+1\}}^{\max\{i,j\}}\abs{\gamma(b_{k})-\gamma(b_{k-1})}\geq \abs{\gamma(b_{j})-\gamma(b_{i})}
		\geq \bConst^{-1} \abs{b_{j}-b_{i}}.
	\end{align}
	For $t\in [b_{j},b_{j+1}]$ and $s\in [b_{i},b_{i+1}]$ holds
	\begin{align}\label{estimatestrgmanycs1}
		\begin{split}
			\MoveEqLeft
			\abs{t-s}\leq \abs{b_{j}-b_{i}}+2\max_{k=1,\ldots,n}\abs{b_{k+1}-b_{k}}\\
			&\stackrel{\text{\eqref{segmentsareapproximatelyequal}}}{\leq}  \abs{b_{j}-b_{i}}+2\bConst\,\frac{\overline c}{c}\min_{k=1,\ldots,n}\abs{a_{k+1}-a_{k}}
			\stackrel{\text{\eqref{inequaltyab}}}{\leq} \Big(1+2\bConst\frac{\overline c}{c}\Big)\abs{b_{j}-b_{i}}.
		\end{split}
	\end{align}
	Additionally, we have
	\begin{align}\label{estimatestrgmanycs}
		\begin{split}
			\MoveEqLeft
			\abs{t-s}\geq \abs{b_{j}-b_{i}}-2\max_{k=1,\ldots,n}\abs{b_{k+1}-b_{k}}\\
			&\stackrel{\text{\eqref{segmentsareapproximatelyequal}}}{\geq} \abs{j-i}\min_{k=1,\ldots,n}\abs{b_{k+1}-b_{k}}-2\bConst\,\frac{\overline c}{c}\min_{k=1,\ldots,n}\abs{b_{k+1}-b_{k}}\\
			&\stackrel{\text{\eqref{segmentsareapproximatelyequal}}}{\geq} \Big(\abs{j-i}-2\bConst\,\frac{\overline c}{c}\Big)\bConst\,\frac{c}{\overline c}\max_{k=1,\ldots,n}\abs{b_{k+1}-b_{k}}
			\geq \bConst\,\frac{c}{2\overline c}\abs{b_{j}-b_{i}}
		\end{split}
	\end{align}
	if $\abs{j-i}\geq N=4\bConst\,\frac{\overline c}{c}$. Furthermore,
	\begin{align}\label{estimatebyfractionalsobolevseminorm}
		\begin{split}
			\MoveEqLeft
			0\leq \abs{t-s}-\abs{\gamma(t)-\gamma(s)}\leq\frac{\abs{t-s}^{2}-\abs{\gamma(t)-\gamma(s)}^{2}}{\abs{t-s}}\\
			&= \frac{\int_{s}^{t}\int_{s}^{t}\abs{\gamma'(v)-\gamma'(u)}^{2}\,\dd u \,\dd v}{2\abs{t-s}}
			\leq \abs{t-s}\int_{s}^{t}\int_{s}^{t}\frac{\abs{\gamma'(v)-\gamma'(u)}^{2}}{\abs{v-u}^{2}}\,\dd u \,\dd v.
		\end{split}
	\end{align}
	\textbf{Step 1}
		Writing $K\vcentcolon=\norm{\gamma''}_{L^{\infty}(\sphere_{L},\R^{d})}$ we estimate
		\begin{align}\label{differencesquareddistances}
			\begin{split}
				\MoveEqLeft
				\abs{t-s}^{2}-\abs{\gamma(t)-\gamma(s)}^{2}=\int_{s}^{t}\int_{s}^{t}(1-\langle \gamma'(u),\gamma'(v) \rangle)\,\dd u\,\dd v\\
				&=\int_{s}^{t}\int_{s}^{t}\Big(1-\Big\langle \gamma'(v)+\int_{v}^{u}\gamma''(\xi)\,\dd \xi,\gamma'(v) \Big\rangle\Big)\,\dd u\,\dd v\\
				&=-\int_{s}^{t}\int_{s}^{t}\int_{v}^{u}\Big\langle\gamma''(\xi),\gamma'(\xi)+\int_{\xi}^{v}\gamma''(\eta)\,\dd\eta \Big\rangle\,\dd \xi\,\dd u\,\dd v\\
				&=-\int_{s}^{t}\int_{s}^{t}\int_{v}^{u}\int_{\xi}^{v}\langle\gamma''(\xi),\gamma''(\eta)\rangle\,\dd\eta\,\dd \xi\,\dd u\,\dd v\\
				&\leq K^{2}\abs{t-s}^{4}.
			\end{split}
		\end{align}
		Then
		\begin{align}\label{localmoebiusenergyestimate}
			\begin{split}
				\MoveEqLeft
				\sum_{i=1}^{n}\sum_{\abs{j-i}\leq N}\int_{b_{j}}^{b_{j+1}}\int_{b_{i}}^{b_{i+1}}\Big(\frac{1}{\abs{\gamma(t)-\gamma(s)}^{2}}-\frac{1}{\abs{t-s}^{2}}\Big)\,\dd s\,\dd t\\
				\leq{}& \bConst^{2}\sum_{i=1}^{n}\sum_{\abs{j-i}\leq N}\int_{b_{j}}^{b_{j+1}}\int_{b_{i}}^{b_{i+1}}\frac{\abs{t-s}^{2}-\abs{\gamma(t)-\gamma(s)}^{2}}{\abs{t-s}^{4}}\,\dd s\,\dd t\\
				\stackrel{\text{\eqref{differencesquareddistances}}}{\leq}{}& \bConst^{2}K^{2}\sum_{i=1}^{n}\sum_{\abs{j-i}\leq N}\abs{b_{j+1}-b_{j}}\abs{b_{i+1}-b_{i}}
				\stackrel{\text{\eqref{segmentsareapproximatelyequal}}}{\leq} \bConst^{4}K^{2}LN\frac{\overline c}{n}.
			\end{split}
		\end{align}
	\textbf{Step 2}
		Now
		\begin{align*}
			\MoveEqLeft
			\sum_{i=1}^{n}\sum_{0<\abs{j-i}\leq N}\Big(\frac{1}{\abs{\gamma(b_{j})-\gamma(b_{i})}^{2}}-\frac{1}{\abs{a_{j}-a_{i}}^{2}}\Big)\abs{\gamma(b_{j+1})-\gamma(b_{j})}\abs{\gamma(b_{i+1})-\gamma(b_{i})}\\
			\stackrel{\text{\eqref{inequaltyab}}}{\leq}{}& \sum_{i=1}^{n}\sum_{0<\abs{j-i}\leq N}\Big(\frac{1}{\abs{\gamma(b_{j})-\gamma(b_{i})}^{2}}-\frac{1}{\abs{b_{j}-b_{i}}^{2}}\Big)
			\abs{\gamma(b_{j+1})-\gamma(b_{j})}\abs{\gamma(b_{i+1})-\gamma(b_{i})}\\
			\leq {}& \bConst^{2}\sum_{i=1}^{n}\sum_{0<\abs{j-i}\leq N}\frac{\abs{b_{j}-b_{i}}^{2}-\abs{\gamma(b_{j})-\gamma(b_{i})}^{2}}{\abs{b_{j}-b_{i}}^{4}}\abs{b_{j+1}-b_{j}}\abs{b_{i+1}-b_{i}}\\
			\stackrel{\text{\eqref{differencesquareddistances}}}{\leq}{}& \bConst^{2}K^{2}\sum_{i=1}^{n}\sum_{0<\abs{j-i}\leq N}\abs{b_{j+1}-b_{j}}\abs{b_{i+1}-b_{i}}
			\stackrel{\text{\eqref{segmentsareapproximatelyequal}}}{\leq}\bConst^{4}K^{2}LN\frac{\overline c}{n}.
		\end{align*}
	\textbf{Step 3}
		From now on let $\abs{j-i}> N$. We have
		\begin{align*}
			\MoveEqLeft[1]
			\Big|\int_{b_{j}}^{b_{j+1}}\int_{b_{i}}^{b_{i+1}}\Big(\frac{\int_{s}^{t}\int_{s}^{t}\langle \gamma'(u),\gamma'(u)-\gamma'(v)\rangle\,\dd u\,\dd v
			-\int_{b_{i}}^{b_{j}}\int_{b_{i}}^{b_{j}}\langle \gamma'(u),\gamma'(u)-\gamma'(v)\rangle\,\dd u\,\dd v}{\abs{\gamma(t)-\gamma(s)}^{2}\abs{t-s}^{2}} \Big)\,\dd s\,\dd t \Big|\\
			&=\vcentcolon A_{i,j}
			= \Big|\int_{b_{j}}^{b_{j+1}}\int_{b_{i}}^{b_{i+1}}\frac{\int_{A}\langle \gamma'(u),\gamma'(u)-\gamma'(v)\rangle\,\dd u\,\dd v
			-\int_{B}\langle \gamma'(u),\gamma'(u)-\gamma'(v)\rangle\,\dd u\,\dd v}{\abs{\gamma(t)-\gamma(s)}^{2}\abs{t-s}^{2}} \,\dd s\,\dd t\Big|
		\end{align*}
		for $A=([s,t]\times[b_{j},t])\cup([b_{j},t]\times[s,t])$, $B=([b_{i},b_{j}]\times[b_{i},s])\cup([b_{i},s]\times[b_{i},b_{j}])$. Now
		\begin{align*}
			\MoveEqLeft
			\Big|\int_{s}^{t}\int_{b_{j}}^{t}\langle \gamma'(v),\gamma'(v)-\gamma'(u) \rangle\,\dd u \,\dd v\Big|\\
			&=\Big|\int_{s}^{t}\int_{b_{j}}^{t}\int_{u}^{v} \Big\langle \gamma'(x)+\int_{x}^{v}\gamma''(y)\,\dd y,\gamma''(x) \Big\rangle \,\dd x\,\dd u \,\dd v\Big|\\
			&=\Big|\int_{s}^{t}\int_{b_{j}}^{t}\int_{u}^{v}\int_{x}^{v} \langle \gamma''(y),\gamma''(x) \rangle \,\dd y\,\dd x\,\dd u \,\dd v\Big|\\
			&\leq K^{2}\abs{t-s}^{3}\abs{t-b_{j}},
		\end{align*}
		can be used to obtain
		\begin{align*}
			\MoveEqLeft
			A_{i,j}\leq 2K^{2}\int_{b_{j}}^{b_{j+1}}\int_{b_{i}}^{b_{i+1}}\frac{\abs{t-s}^{3}\abs{t-b_{j}}+\abs{b_{j}-b_{i}}^{3}\abs{s-b_{i}}}{\abs{\gamma(t)-\gamma(s)}^{2}\abs{t-s}^{2}} \,\dd s\,\dd t\\
			\stackrel{\text{\eqref{estimatestrgmanycs1},\,\eqref{estimatestrgmanycs}}}{\leq} {}& \gConst \frac{(\max_{k=1,\ldots,n}\abs{b_{k+1}-b_{k}})^{3}}{\abs{b_{j}-b_{i}}}
			\stackrel{\text{\eqref{segmentsareapproximatelyequal}}}{\leq} \gConst \frac{1}{\abs{j-i}n^{2}},
		\end{align*}
		where $\gConst$ is a constant that may change from line to line.\\
	\textbf{Step 4}
		Now,
		\begin{align*}
			\MoveEqLeft
			\Big|\abs{\gamma(b_{j})-\gamma(b_{i})}^{2}-\abs{\gamma(t)-\gamma(s)}^{2}\Big|\\
			= {}&\Big(\abs{\gamma(b_{j})-\gamma(b_{i})}+\abs{\gamma(t)-\gamma(s)}\Big)\Big|\abs{\gamma(b_{j})-\gamma(b_{i})}-\abs{\gamma(t)-\gamma(s)}\Big|\\
			\stackrel{\text{\eqref{estimatestrgmanycs1}}}{\leq} {}& \gConst\abs{b_{j}-b_{i}}\Big(\abs{\gamma(b_{j})-\gamma(t)}+\abs{-\gamma(b_{i})+\gamma(s)}\Big)\\
			\leq {}& \gConst\abs{b_{j}-b_{i}}\max_{k=1,\ldots,n}\abs{b_{k+1}-b_{k}}
		\end{align*}
		and similarly $\big|\abs{b_{j}-b_{i}}^{2}-\abs{t-s}^{2}\big|
			\stackrel{\text{\eqref{estimatestrgmanycs1}}}{\leq} \gConst\abs{b_{j}-b_{i}}\max_{k=1,\ldots,n}\abs{b_{k+1}-b_{k}}$.
		Putting $A=\abs{\gamma(b_{j})-\gamma(b_{i})}$, $B=\abs{b_{j}-b_{i}}$, $a=\abs{\gamma(t)-\gamma(s)}$ and $b=\abs{t-s}$ we find
		\begin{align}\label{estimateaandbs}
			\abs{A^{2}B^{2}-a^{2}b^{2}}\leq \abs{A^{2}-a^{2}}B^{2}+a^{2}\abs{B^{2}-b^{2}}
			\stackrel{\text{\eqref{estimatestrgmanycs1}}}{\leq} \gConst\abs{b_{j}-b_{i}}^{3}\max_{k=1,\ldots,n}\abs{b_{k+1}-b_{k}}.
		\end{align}
		Therefore,
		\begin{align*}
			\MoveEqLeft
			B_{i,j}\vcentcolon=\Big|\int_{b_{j}}^{b_{j+1}}\int_{b_{i}}^{b_{i+1}}\Big(\abs{b_{j}-b_{i}}^{2}-\abs{\gamma(b_{j})-\gamma(b_{i})}^{2}\Big)\\
			&\qquad\Big(\frac{1}{\abs{\gamma(t)-\gamma(s)}^{2}\abs{t-s}^{2}}-\frac{1}{\abs{\gamma(b_{j})-\gamma(b_{i})}^{2}\abs{b_{j}-b_{i}}^{2}}\Big)\,\dd s\,\dd t \Big|\\
			\stackrel{\text{\eqref{differencesquareddistances}}}{\leq} {}& K^{2}\abs{b_{j}-b_{i}}^{4}
			\int_{b_{j}}^{b_{j+1}}\int_{b_{i}}^{b_{i+1}}\frac{\abs{\abs{\gamma(b_{j})-\gamma(b_{i})}^{2}\abs{b_{j}-b_{i}}^{2}-\abs{\gamma(t)-\gamma(s)}^{2}\abs{t-s}^{2}}}{\abs{\gamma(t)-\gamma(s)}^{2}\abs{t-s}^{2}\abs{\gamma(b_{j})-\gamma(b_{i})}^{2}\abs{b_{j}-b_{i}}^{2}}\,\dd s\,\dd t\\
			\stackrel{\text{\eqref{estimateaandbs}}}{\leq} {}& \gConst \abs{b_{j}-b_{i}}^{4}
			\int_{b_{j}}^{b_{j+1}}\int_{b_{i}}^{b_{i+1}}\frac{\abs{b_{j}-b_{i}}^{3}\max_{k=1,\ldots,n}\abs{b_{k+1}-b_{k}}}{\abs{\gamma(t)-\gamma(s)}^{2}\abs{t-s}^{2}\abs{\gamma(b_{j})-\gamma(b_{i})}^{2}\abs{b_{j}-b_{i}}^{2}}\,\dd s\,\dd t\\
			\stackrel{\text{\eqref{estimatestrgmanycs}}}{\leq} {}& \gConst \abs{b_{j}-b_{i}}^{4}
			\int_{b_{j}}^{b_{j+1}}\int_{b_{i}}^{b_{i+1}}\frac{\abs{b_{j}-b_{i}}^{3}\max_{k=1,\ldots,n}\abs{b_{k+1}-b_{k}}}{\abs{b_{j}-b_{i}}^{8}}\,\dd s\,\dd t\\
			\leq \,{}& \gConst \frac{(\max_{k=1,\ldots,n}\abs{b_{k+1}-b_{k}})^{3}}{\abs{b_{j}-b_{i}}}
			\stackrel{\text{\eqref{segmentsareapproximatelyequal}}}{\leq} \gConst \frac{1}{\abs{j-i}n^{2}}.
		\end{align*}
			Since $\sum_{k=1}^{n}\frac{1}{k}-\log(n)$ converges to the Euler-Mascheroni constant, we obtain
		\begin{align*}
			\sum_{\substack{i,j=1\\\abs{j-i}>N}}^{n}(A_{i,j}+B_{i,j})\leq \gConst \frac{1}{n^{2}}\sum_{i=1}^{n}\sum_{k=1}^{n}\frac{1}{k}
			= \gConst \frac{1}{n^{1-\epsilon}}\frac{1}{n^{\epsilon}}\sum_{k=1}^{n}\frac{1}{k}
			\leq \gConst \frac{1}{n^{1-\epsilon}}.
		\end{align*}
	\textbf{Step 5}
		Set
		\begin{align*}
			\MoveEqLeft
			C_{i,j}\vcentcolon=\Big|\Big(\frac{1}{\abs{\gamma(b_{j})-\gamma(b_{i})}^{2}}-\frac{1}{\abs{b_{j}-b_{i}}^{2}}\Big)\abs{b_{j+1}-b_{j}}\abs{b_{i+1}-b_{i}} \\
			&\qquad -\Big(\frac{1}{\abs{\gamma(b_{j})-\gamma(b_{i})}^{2}}-\frac{1}{\abs{b_{j}-b_{i}}^{2}}\Big)\abs{\gamma(b_{j+1})-\gamma(b_{j})}\abs{\gamma(b_{i+1})-\gamma(b_{i})}\Big|.
		\end{align*}
		As in \eqref{differencesquareddistances} we estimate
		\begin{align*}
			\Big|\frac{1}{\abs{\gamma(b_{j})-\gamma(b_{i})}^{2}}-\frac{1}{\abs{b_{j}-b_{i}}^{2}}\Big|
			=\frac{\abs{b_{j}-b_{i}}^{2}-\abs{\gamma(b_{j})-\gamma(b_{i})}^{2}}{\abs{\gamma(b_{j})-\gamma(b_{i})}^{2}\abs{b_{j}-b_{i}}^{2}}
			\leq \bConst^{2}K^{2}.
		\end{align*}
		Putting $A=\abs{b_{i+1}-b_{i}}$, $a=\abs{\gamma(b_{i+1})-\gamma(b_{i})}$, $B=\abs{b_{j+1}-b_{j}}$, $b=\abs{\gamma(b_{j+1})-\gamma(b_{j})}$ we have
		\begin{align*}
			\frac{A}{\bConst}\leq a\leq A,\quad \frac{B}{\bConst}\leq b\leq B\quad\text{and}\quad \frac{A}{B}\leq \bConst\frac{\overline c}{c},\quad\frac{B}{A}\leq \bConst\frac{\overline c}{c}
		\end{align*}
		at our disposal and can estimate the remaining factor by
		\begin{align*}
			\MoveEqLeft
			AB-ab
			=(A-a)B+a(B-b)
			=\frac{A^{2}-a^{2}}{A+a}B+a\frac{B^{2}-b^{2}}{B+b}
			\leq \frac{A^{2}-a^{2}}{A+\frac{A}{\bConst}}B+A\frac{B^{2}-b^{2}}{B+\frac{B}{\bConst}}\\
			&\leq \frac{\bConst\frac{\overline c}{c}}{1+\frac{1}{\bConst}}(A^{2}-a^{2}+B^{2}-b^{2})
			\stackrel{\text{\eqref{differencesquareddistances}}}{\leq} \frac{2\bConst\frac{\overline c}{c}K^{2}}{1+\frac{1}{\bConst}}(\max\{A,B\})^{4}.
		\end{align*}
		Hence, $\sum_{\substack{i,j=1,\,i\not=j}}^{n}C_{i,j}\leq n^{2}\gConst(\max\{A,B\})^{4}\stackrel{\text{\eqref{segmentsareapproximatelyequal}}}{\leq} \gConst\frac{1}{n^{2}}.$\\
	\textbf{Step 6}
		Without loss of generality we assume $i<j$. Then
		\begin{align*}
			\MoveEqLeft
			\abs{b_{j}-b_{i}}^{2}-\abs{a_{j}-a_{i}}^{2}
			=(\abs{b_{j}-b_{i}}+\abs{a_{j}-a_{i}})\Big(\sum_{k=i}^{j-1}\abs{b_{k+1}-b_{k}}-\sum_{k=i}^{j-1}\abs{a_{k+1}-a_{k}}\Big)\\
			&\stackrel{\text{\eqref{inequaltyab},\eqref{estimatebyfractionalsobolevseminorm}}}{\leq}2K^{2}\abs{b_{j}-b_{i}}^{2}\Big(\max_{k=i,\ldots,j-1}\abs{b_{k+1}-b_{k}}\Big)^{2}.
		\end{align*}
		Hence,
		\begin{align*}
			\MoveEqLeft
			0\stackrel{\text{\eqref{inequaltyab}}}{<}D_{i,j}\vcentcolon=\Big(\frac{1}{\abs{a_{j}-a_{i}}^{2}}-\frac{1}{\abs{b_{j}-b_{i}}^{2}}\Big)\abs{a_{j+1}-a_{j}}\abs{a_{i+1}-a_{i}}\\
			\leq{}& \frac{2K^{2}(\max_{k=i,\ldots,j-1}\abs{b_{k+1}-b_{k}})^{2}}{\abs{a_{j}-a_{i}}^{2}}\abs{a_{j+1}-a_{j}}\abs{a_{i+1}-a_{i}}\\
			\stackrel{\text{\eqref{inequaltyab}}}{\leq}{}&\frac{2K^{2}\bConst^{2}(\max_{k=1,\ldots,n}\abs{a_{k+1}-a_{k}})^{4}}{\abs{a_{j}-a_{i}}^{2}}
			\leq \frac{2K^{2}\bConst^{2}(\max_{k=1,\ldots,n}\abs{a_{k+1}-a_{k}})^{4}}{(\abs{j-i}\min_{k=1,\ldots,n}\abs{a_{k+1}-a_{k}})^{2}}\\
			\stackrel{\text{\eqref{segmentsareapproximatelyequal}}}{\leq}{}& \frac{2K^{2}\bConst^{2}(\frac{\overline c}{c})^{2}}{\abs{j-i}^{2}}\max_{k=1,\ldots,n}\abs{a_{k+1}-a_{k}}^{2}
			\stackrel{\text{\eqref{segmentsareapproximatelyequal}}}{\leq} \frac{2K^{2}\bConst^{2}\overline c^{4}}{c^{2}n^{2}\abs{j-i}^{2}}
		\end{align*}
		and since $\zeta(2)=\frac{\pi^{2}}{6}$ we obtain $0<\sum_{\substack{i,j=1,\,i\not=j}}^{n}D_{i,j}\leq \gConst\frac{1}{n^{2}}n\sum_{k=1}^{n-1}\frac{1}{k^{2}}\leq \gConst \frac{1}{n}$.
\end{proof}

\begin{specialcorollary*}[Convergence of M\"obius energies of inscribed polygons]
	Let $\gamma\in\mathcal{C}$ with $\mathcal{E}(\gamma)<\infty$ and $p_{n}$ as in Proposition \ref{orderofconvergence}. Then $\lim_{n\to\infty}\mathcal{E}_{n}(p_{n})=\mathcal{E}(\gamma)$.
\end{specialcorollary*}
\begin{proof}
	Set $\epsilon_{n}\vcentcolon=\sum_{i=1}^{n}\int_{b_{k}}^{b_{k+1}}\int_{b_{k}}^{b_{k+1}}\frac{\abs{\gamma'(v)-\gamma'(u)}^{2}}{\abs{v-u}^{2}}\,\dd u\,\dd v$
	and $N_{n}\vcentcolon= n\max\{\epsilon_{n}^{\frac{1}{4}},n^{-\frac{1}{6}}\}$.
	Then
	\begin{align}\label{estimatemeasurebs}
		\HM^{2}\Big(\bigcup_{\abs{j-i}\leq N_{n}+1}[b_{j},b_{j+1}]\times [b_{i},b_{i+1}]\Big)
		\stackrel{\text{\eqref{inequaltyab}}}{\leq} n 4N_{n}\frac{\bConst^{2}\overline{c}^{2}}{n^{2}}
		= 4\bConst^{2}\overline{c}^{2} \max\{\epsilon_{n}^{\frac{1}{4}},n^{-\frac{1}{6}}\}\xrightarrow[n\to\infty]{} 0.
	\end{align}
	\textbf{Step 1}
		According to \eqref{estimatebyfractionalsobolevseminorm} the local bi-Lipschitz constant can be uniformly chosen as close to $1$ as we wish, 
		so that for $i<j$, $\delta=\frac{1}{1+2\frac{\overline c}{c}}$ and $n\geq M_{\delta}$ we can use \eqref{segmentsareapproximatelyequal} to find
		\begin{align*}
			\MoveEqLeft
			\abs{b_{j+1}-b_{j}}\leq (1-\delta+2\delta)\abs{\gamma(b_{j+1})-\gamma(b_{j})}
			\leq(1-\delta)(\abs{\gamma(b_{j+1})-\gamma(b_{j})}+\abs{\gamma(b_{j})-\gamma(b_{j-1})})\\
			&\leq (1-\delta)(\abs{b_{j+1}-b_{j}}+\abs{b_{j}-b_{j-1}})=(1-\delta)\abs{b_{j+1}-b_{j-1}}
		\end{align*}
		and thus
		\begin{align}\label{comparebs}
			\delta\abs{b_{j+1}-b_{i}}\leq \abs{b_{j+1}-b_{i}}-(1-\delta)\abs{b_{j+1}-b_{j-1}}\leq \abs{b_{j+1}-b_{i}}-\abs{b_{j+1}-b_{j}}
			\leq \abs{b_{j}-b_{i}}.
		\end{align}
		Hence, 
		\begin{align}\label{localestimatediscretemoebiusenergy}
			\begin{split}
				\MoveEqLeft
				\sum_{i=1}^{n}\sum_{0<j-i\leq N_{n}}\Big(\frac{1}{\abs{\gamma(b_{j})-\gamma(b_{i})}^{2}}-\frac{1}{\abs{a_{j}-a_{i}}^{2}}\Big)\abs{\gamma(b_{j+1})-\gamma(b_{j})}\abs{\gamma(b_{i+1})-\gamma(b_{i})}\\
				\stackrel{\text{\eqref{segmentsareapproximatelyequal},\eqref{inequaltyab}}}{\leq} {}& \frac{\overline c}{c}\sum_{i=1}^{n}\sum_{1<j-i\leq N_{n}}
				\Big(\frac{1}{\abs{\gamma(b_{j})-\gamma(b_{i})}^{2}}-\frac{1}{\abs{b_{j}-b_{i}}^{2}}\Big)\abs{\gamma(b_{j+1})-\gamma(b_{j})}\abs{\gamma(b_{i})-\gamma(b_{i-1})}\\
				\leq\;\;\,{}& \bConst^{2}\frac{\overline c}{c}\sum_{i=1}^{n}\sum_{1<j-i\leq N_{n}}\frac{\abs{b_{j}-b_{i}}^{2}-\abs{\gamma(b_{j})-\gamma(b_{i})}^{2}}{\abs{b_{j}-b_{i}}^{4}}\abs{b_{j+1}-b_{j}}\abs{b_{i}-b_{i-1}}\\
				=\;\;\, {}& \bConst^{2}\frac{\overline c}{c}\sum_{i=1}^{n}\sum_{1<j-i\leq N_{n}}\int_{b_{j}}^{b_{j+1}}\int_{b_{i-1}}^{b_{i}}\frac{\int_{b_{i}}^{b_{j}}\int_{b_{i}}^{b_{j}}\abs{\gamma'(v)-\gamma'(u)}^{2}\,\dd u\,\dd v}{2\abs{b_{j}-b_{i}}^{4}}\,\dd s\,\dd t\\
				\stackrel{\text{\eqref{comparebs}}}{\leq}\;\;
				 {}& \delta^{-8}\bConst^{2}\frac{\overline c}{c}\sum_{i=1}^{n}\sum_{1<j-i\leq N_{n}}\int_{b_{j}}^{b_{j+1}}\int_{b_{i-1}}^{b_{i}}\frac{\int_{s}^{t}\int_{s}^{t}\abs{\gamma'(v)-\gamma'(u)}^{2}\,\dd u\,\dd v}{2\abs{t-s}^{4}}\,\dd s\,\dd t\\
				 \leq\;\;\,{}&\delta^{-8}\bConst^{2}\frac{\overline c}{c}\sum_{i=1}^{n}\sum_{1<j-i\leq N_{n}}
				 \int_{b_{j}}^{b_{j+1}}\int_{b_{i-1}}^{b_{i}}\Big(\frac{1}{\abs{\gamma(t)-\gamma(s)}^{2}}-\frac{1}{\abs{t-s}^{2}}\Big)\,\dd s\,\dd t.
			\end{split}
		\end{align}
		The sum over $1<\abs{j-i}\leq N_{n}$ can be estimated analogously.
		According to \eqref{estimatemeasurebs} we know that the left-hand side in \eqref{localmoebiusenergyestimate} with $N_{n}$ instead of $N$ as well as \eqref{localestimatediscretemoebiusenergy} 
		converge to zero for $n\to\infty$.\\		
	\textbf{Step 2}
		If $n$ is large enough for $\abs{j-i}\geq N_{n}$ holds $\abs{\gamma(b_{j})-\gamma(b_{i})}\geq \gConst\frac{N_{n}}{n}$ by \eqref{segmentsareapproximatelyequal}.
		We estimate
		\begin{align*}
			A_{i,j}&\leq \gConst\int_{b_{j}}^{b_{j+1}}\int_{b_{i}}^{b_{i+1}}\frac{2\frac{1}{n}\abs{b_{j+1}-b_{i}}}{\abs{\gamma(t)-\gamma(s)}^{2}\abs{t-s}^{2}}\,\dd s\,\dd t
			\leq \gConst \frac{1}{n^{2}}\frac{1}{n}\Big(\frac{n}{N_{n}}\Big)^{3}
			\leq \gConst \frac{1}{n^{\frac{1}{2}}}\frac{1}{n^{2}},\\
			B_{i,j}&\stackrel{\text{\eqref{estimatebyfractionalsobolevseminorm}}}{\leq} \gConst\frac{1}{n^{2}}\abs{b_{j}-b_{i}}^{2}\epsilon_{n}\frac{1}{\abs{b_{j}-b_{i}}^{4}}
			\leq \gConst \frac{1}{n^{2}}\epsilon_{n}\frac{n^{2}}{N_{n}^{2}}
			\leq \gConst \epsilon_{n}^{\frac{1}{2}}\frac{1}{n^{2}},\\
			C_{i,j}&\stackrel{\text{\eqref{estimatebyfractionalsobolevseminorm}}}{\leq}\gConst \frac{n^{2}}{N_{n}^{2}}\epsilon_{n}\frac{1}{n^{2}}
			\leq \gConst \epsilon_{n}^{\frac{1}{2}}\frac{1}{n^{2}},\\
			D_{i,j}&\stackrel{\text{\eqref{estimatebyfractionalsobolevseminorm}}}{\leq} \gConst \epsilon_{n}\frac{\abs{b_{j}-b_{i}}^{2}}{\abs{a_{j}-a_{i}}^{2}\abs{b_{j}-b_{i}}^{2}}\frac{1}{n^{2}}
			\stackrel{\text{\eqref{inequaltyab}}}{\leq} \gConst \epsilon_{n}\frac{n^{2}}{N_{n}^{2}}\frac{1}{n^{2}}
			\leq \gConst \epsilon_{n}^{\frac{1}{2}}\frac{1}{n^{2}}.
		\end{align*}
		And thus $\sum_{N_{n}<\abs{j-i}}(A_{i,j}+B_{i,j}+C_{i,j}+D_{i,j})\to 0$ for $n\to\infty$.
\end{proof}

\section{The $\limsup$ inequality}

\begin{proposition}[The $\limsup$ inequality]\label{limsupmoebius}
	Let $\gamma\in \mathcal{C}(\mathcal{K})\cap C^{1}$ with $\mathcal{E}(\gamma)<\infty$. Then there are $p_{n}\in \mathcal{P}_{n}(\mathcal{K})$ such that
	\begin{align*}
		p_{n}\xrightarrow[n\to\infty]{W^{1,\infty}(\sphere_{1},\R^{d})} \gamma \qquad\text{and}\qquad
		\lim_{n\to\infty}\mathcal{E}_{n}(p_{n})=\mathcal{E}(\gamma).
	\end{align*}
\end{proposition}
\begin{proof}
	\textbf{Step 1}
		The mapping $(u,v)\mapsto\langle \gamma'(u),\gamma'(v) \rangle$ is uniformly continuous, and there is $d>0$ such that for all $x\in\sphere_{1}$ and all $u,v\in B_{3d}(x)$ holds
		$\langle \gamma'(u),\gamma'(v) \rangle \geq \cos(\frac{1}{4})$. Therefore, for all $x\in\sphere_{1}$ and all $s,t,y\in B_{3d}(x)$ holds
		\begin{align*}
			\MoveEqLeft
			\abs{\langle \gamma(s)-\gamma(t),\gamma(x)-\gamma(y) \rangle }=\int_{[s,t]\cup[t,s]}\int_{[x,y]\cup[y,x]}\langle \gamma'(u),\gamma'(v) \rangle\,\dd u\,\dd v\\
			&\geq \cos(\tfrac{1}{4})\abs{\gamma(s)-\gamma(t)}\abs{\gamma(x)-\gamma(y)},
		\end{align*}
		which means that $\gamma$ has the $(d,\frac{1}{4})$ diamond property, as defined in \cite[Definition 4.5]{Strzelecki2013b}. According to \cite[Theorem 4.10]{Strzelecki2013b}
		inscribed polygons of edge length smaller than $d$ belong to the same knot class as $\gamma$.
		By \cite{Wu2004a} we know that for each $n$ there is a closed equilateral polygon $\tilde p$ with $n$ edges that is inscribed in $\gamma$, so that this polygon belongs to the same knot class as $\gamma$
		if $n$ is large enough.\\
	\textbf{Step 2}
		Let $\tilde p$ be an equilateral inscribed polygon with $n$ edges defined by $\gamma(b_{i})$, $i=0,\ldots, n$, $\gamma(b_{0})=\gamma(b_{n})$ with $b_{0}=0$ and $n$ sufficiently large. 
		Then for $\epsilon>0$ and $n\geq N(\epsilon)$ holds
		\begin{align}\label{lengthestimateinscribedpolygon}
			\begin{split}
				\MoveEqLeft
				0\leq \mathcal{L}(\gamma)-\mathcal{L}(\tilde p)=\sum_{i=0}^{n-1}(\abs{b_{i+1}-b_{i}}-\abs{\gamma(b_{i+1})-\gamma(b_{i})})\\
				&\stackrel{\text{\eqref{estimatebyfractionalsobolevseminorm}}}{\leq} \max_{i=1,\ldots,n}\abs{b_{i+1}-b_{i}}\sum_{i=0}^{n-1}\int_{b_{i}}^{b_{i+1}}\int_{b_{i}}^{b_{i+1}}\frac{\abs{\gamma'(v)-\gamma'(u)}^{2}}{\abs{v-u}^{2}}\,\dd u \,\dd v
				\leq \frac{\bConst\,\epsilon}{n}.
			\end{split}
		\end{align}
	\textbf{Step 3}
		Set $p(t)=L\tilde L^{-1}\tilde p(\tilde L L^{-1}t)$. For $t\in[a_{j},a_{j+1}]$, $a_{j}=\frac{jL}{n}$ and a constant $c=c(\gamma)$ holds
		\begin{align*}
			\MoveEqLeft
			\abs{\gamma(t)-p(t)}\leq \abs{\gamma(t)-\gamma(b_{j+1})}+\abs{\gamma(b_{j+1})-L\tilde L^{-1}\gamma(b_{j+1})}+\abs{p(a_{j+1})-p(t)}\\
			&\leq \abs{t-a_{j+1}}+\abs{a_{j+1}-b_{j+1}}+\tilde L^{-1}\abs{L-\tilde L}\norm{\gamma}_{L^{\infty}(\sphere_{1},\R^{d})}+\abs{a_{j+1}-t}\leq \frac{c}{n}
		\end{align*}
		due to
		\begin{align}\label{differencebandaj}
			\begin{split}
				\MoveEqLeft
				\abs{b_{j}-a_{j}}=\Big|\sum_{k=0}^{j-1}\abs{b_{k+1}-b_{k}}-\sum_{k=0}^{j-1}\frac{L}{\tilde L}\abs{\gamma(b_{k+1})-\gamma(b_{k})}\Big|\\
				&\leq \sum_{k=0}^{j-1}\Big|\abs{b_{k+1}-b_{k}}-\abs{\gamma(b_{k+1})-\gamma(b_{k})}\Big|
				+\Big|1-\frac{L}{\tilde L}\Big|\sum_{k=0}^{j-1}\abs{\gamma(b_{k+1})-\gamma(b_{k})}\Big|\\
				&\leq 2\abs{L-\tilde L}\stackrel{\text{\eqref{lengthestimateinscribedpolygon}}}{\leq} \frac{2\bConst\,\epsilon}{n}.
			\end{split}
		\end{align}
	\textbf{Step 4}
		For $t\in[a_{j},a_{j+1}]$ we estimate
		\begin{align*}
			\MoveEqLeft
			\abs{\gamma'(t)-p'(t)}^{2}
			=\Big| \gamma'(t)-\frac{\int_{b_{j}}^{b_{j+1}}\gamma'(s)\,\dd s}{\abs{\gamma(b_{j+1})-\gamma(b_{j})}} \Big|^{2}
			=2\Big(1-\frac{\int_{b_{j}}^{b_{j+1}}\langle \gamma'(t),\gamma'(s) \rangle\,\dd s}{\abs{\gamma(b_{j+1})-\gamma(b_{j})}}\Big)\\
			&=2\frac{\abs{\gamma(b_{j+1})-\gamma(b_{j})}-\int_{b_{j}}^{b_{j+1}}\langle \gamma'(t),\gamma'(s) \rangle\,\dd s}{\abs{\gamma(b_{j+1})-\gamma(b_{j})}}
			=2\frac{\int_{b_{j}}^{b_{j+1}}(1-\langle \gamma'(t),\gamma'(s) \rangle)\,\dd s}{\abs{b_{j+1}-b_{j}}}\\
			&=\frac{\int_{b_{j}}^{b_{j+1}}\abs{\gamma'(t) -\gamma'(s) }^{2}\,\dd s}{\abs{b_{j+1}-b_{j}}}
			\leq \epsilon^{2}
		\end{align*}
		if $n$ is large enough, due to the uniform continuity of $\gamma'$ and \eqref{differencebandaj}.\\
	\textbf{Step 5}
		Since the discrete M\"obius energy is invariant under scaling, the proposition is a consequence of Corollary \ref{convergencemoebiusinscribedpolygons}.
\end{proof}

Note, that, by integrating the inequality in Step 4 instead of using continuity of $\gamma'$, we easily find that for $\gamma\in W^{1+\frac{1}{2},2}(\sphere_{1},\R^{d})$
the rescaled inscribed polygons converge in $W^{1,2}$.

\section{Discrete minimizers}

\begin{speciallemma*}[Regular $n$-gon is unique minimizer of $\mathcal{E}_{n}$ in $\mathcal{P}_{n}$]
	The unique minimizer of $\mathcal{E}_{n}$ in $\mathcal{P}_{n}$ is a regular $n$-gon.
\end{speciallemma*}
\begin{proof}
	Using the inequality of arithmetic and geometric means twice we obtain
	\begin{align*}
		\sum_{i=1}^{n}\frac{1}{\abs{p(a_{i})-p(a_{i+k})}^{2}}\geq n\Big(\prod_{i=1}^{n}\frac{1}{\abs{p(a_{i})-p(a_{i+k})}^{2}}\Big)^{\frac{1}{n}}
		\geq \frac{n^{2}}{\sum_{i=1}^{n}\abs{p(a_{i})-p(a_{i+k})}^{2}},
	\end{align*}
	with equality if and only if all $\abs{p(a_{i})-p(a_{i+k})}$ are equal. From \cite[Theorem III]{Gabor1966a} we know that for $n\geq 4$ the sum of diagonals of an equilateral polygon
	is maximized by the regular $n$-gon $g_{n}$, i.e. 
	\begin{align*}
		\sum_{i=1}^{n}\abs{p(a_{i})-p(a_{i+k})}^{2}\leq \sum_{i=1}^{n}\abs{g_{n}(a_{i})-g_{n}(a_{i+k})}^{2}.
	\end{align*}
	Note, that this also works in $\R^{d}$, thanks to \cite[Lemma 7]{Abrams2003a}, with equality for fixed $k\in\{2,\ldots,n-2\}$ if and only if for all $i$ 
	the points $p(a_{i}), p(a_{i+1}),p(a_{i+k})$ and $p(a_{i+k+1})$ are coplanar. This yields
	\begin{align*}
		\MoveEqLeft
		\sum_{k=1}^{n-1}\sum_{i=1}^{n}\frac{1}{\abs{p(a_{i})-p(a_{i+k})}^{2}}\geq \sum_{k=1}^{n-1}\frac{n^{2}}{\sum_{i=1}^{n}\abs{p(a_{i})-p(a_{i+k})}^{2}}\\
		&\geq \sum_{k=1}^{n-1}\frac{n^{2}}{\sum_{i=1}^{n}\abs{g_{n}(a_{i})-g_{n}(a_{i+k})}^{2}}=\sum_{k=1}^{n-1}\sum_{i=1}^{n}\frac{1}{\abs{g_{n}(a_{i})-g_{n}(a_{i+k})}^{2}},
	\end{align*}
	with equality if and only if $p$ is a planar polygon, which follows from the coplanarity before, that is the affine image of a regular polygon, see \cite[Theorem III]{Gabor1966a}, 
	such that all diagonals of the same order have equal length. This means, equality only holds for a regular $n$-gon.
\end{proof}

\begin{appendix}

\section{Variational convergence}

\begin{lemma}[Convergence of minimizers]\label{variationalconvergencegeneral}
	Let $\mathcal{F}_{n},\mathcal{F}:X\to\overline\R$, $Y\subset X$. Assume that $x_{n}\to x$ implies $\mathcal{F}(x)\leq \liminf_{n\to\infty}\mathcal{F}_{n}(x_{n})$ and that for
	every $y\in Y$ there is are $y_{n}\in X$ with $\limsup_{n\to\infty}\mathcal{F}_{n}(y_{n})\leq\mathcal{F}(y)$.
	Let $\abs{\mathcal{F}_{n}(z_{n})-\inf_{X}\mathcal{F}_{n}}\to 0$ and $z_{n}\to z\in X$. Then $\mathcal{F}(z)\leq\liminf_{n\to\infty}\inf_{X}\mathcal{F}_{n}\leq \inf_{Y}\mathcal{F}$.
\end{lemma}
\begin{proof}
	Let $y\in Y$ and $y_{n}\to y$ with $\limsup_{n\to\infty}\mathcal{F}_{n}(y_{n})\leq\mathcal{F}(y)$. Then
	\begin{align*}
		\mathcal{F}(z)\leq \liminf_{n\to\infty}\mathcal{F}_{n}(z_{n})=\liminf_{n\to\infty}\inf_{X}\mathcal{F}_{n}\
		\leq \liminf_{n\to\infty}\mathcal{F}_{n}(y_{n})\leq \limsup_{n\to\infty}\mathcal{F}_{n}(y_{n})\leq \mathcal{F}(y).
	\end{align*}
\end{proof}

\end{appendix}

\bibliography{/Users/sebastianscholtes/Documents/library.bib}{}
\bibliographystyle{amsalpha}
\bigskip
\noindent
\parbox[t]{.8\textwidth}{
Sebastian Scholtes\\
Institut f{\"u}r Mathematik\\
RWTH Aachen University\\
Templergraben 55\\
D--52062 Aachen, Germany\\
sebastian.scholtes@rwth-aachen.de}

\end{document}